\documentclass[12pt,a4paper]{amsart}
\usepackage{amsfonts}
\usepackage{amsthm}
\usepackage{eqnarray,amsmath}
 \usepackage{relsize}
\usepackage{amscd}
\usepackage[latin2]{inputenc}
\usepackage{t1enc}    
\usepackage{lipsum}
\newcommand{\1}{\mathbf{1}}

\usepackage[mathscr]{eucal}
\usepackage{indentfirst}
\usepackage{graphicx}
\usepackage{nicefrac}
\usepackage{eqnarray}
\usepackage{graphics}
\usepackage{pict2e}
\usepackage{epic}
\numberwithin{equation}{section}
\usepackage[margin=2.9cm]{geometry}
\usepackage{epstopdf} 

\newcommand{\var}{\text{Var}}

\theoremstyle{plain}
\newtheorem{thm}{Theorem}[section]
\newtheorem{lemma}[thm]{Lemma}

\newtheorem{prop}[thm]{Proposition}

\theoremstyle{definition}

\newtheorem{remark}[thm]{Remark}

\theoremstyle{remark}

\newtheorem*{stat*}{Statement}
\newtheorem*{prop*}{Proposition}
\numberwithin{equation}{section}

\newcommand{\E}{\mathbb{E}}
\newcommand{\of}[1]{\left(#1\right)}
\renewcommand{\Im}{\mathfrak{Im}}
\renewcommand{\Re}{\mathfrak{Re}}

\newcommand{\N}{\mathbb{N}}
\newcommand{\Z}{\mathbb{Z}}

\newcommand{\R}{\mathbb{R}}
\newcommand{\C}{\mathbb{C}}
\newcommand{\T}{\mathbb{T}}

\newcommand{\Tr}{\operatorname{Tr}}
\usepackage{hyperref}
\hypersetup{
    colorlinks=true,
    linkcolor=blue,
    urlcolor=blue,
    linktoc=all
}

\makeatletter
\def\chaptermark#1{}%whatever

\def\chapter{%
  \if@openright\cleardoublepage\else\clearpage\fi
  \thispagestyle{plain}\global\@topnum\z@
  \@afterindenttrue \secdef\@chapter\@schapter}

\def\@chapter[#1]#2{\refstepcounter{chapter}%
  \ifnum\c@secnumdepth<\z@ \let\@secnumber\@empty
  \else \let\@secnumber\thechapter \fi
  \typeout{\chaptername\space\@secnumber}%
  \def\@toclevel{0}%
  \ifx\chaptername\appendixname \@tocwriteb\tocappendix{chapter}{#2}%
  \else \@tocwriteb\tocchapter{chapter}{#2}\fi
  \chaptermark{#1}%
  \addtocontents{lof}{\protect\addvspace{10\p@}}%
  \addtocontents{lot}{\protect\addvspace{10\p@}}%
  \@makechapterhead{#2}\@afterheading}
\def\@schapter#1{\typeout{#1}%
  \let\@secnumber\@empty
  \def\@toclevel{0}%
  \ifx\chaptername\appendixname \@tocwriteb\tocappendix{chapter}{#1}%
  \else \@tocwriteb\tocchapter{chapter}{#1}\fi
  \chaptermark{#1}%
  \addtocontents{lof}{\protect\addvspace{10\p@}}%
  \addtocontents{lot}{\protect\addvspace{10\p@}}%
  \@makeschapterhead{#1}\@afterheading}
\newcommand\chaptername{Chapter}

\def\@makechapterhead#1{\global\topskip 7.5pc\relax
  \begingroup
  \fontsize{\@xivpt}{18}\bfseries\centering
    \ifnum\c@secnumdepth>\m@ne
      \leavevmode \hskip-\leftskip
      \rlap{\vbox to\z@{\vss
          \centerline{\normalsize\mdseries
              \uppercase\@xp{\chaptername}\enspace\thechapter}
          \vskip 3pc}}\hskip\leftskip\fi
     #1\par \endgroup
  \skip@34\p@ \advance\skip@-\normalbaselineskip
  \vskip\skip@ }
\def\@makeschapterhead#1{\global\topskip 7.5pc\relax
  \begingroup
  \fontsize{\@xivpt}{18}\bfseries\centering
  #1\par \endgroup
  \skip@34\p@ \advance\skip@-\normalbaselineskip
  \vskip\skip@ }
\def\appendix{\par
  \c@chapter\z@ \c@section\z@
  \let\chaptername\appendixname
  \def\thechapter{\@Alph\c@chapter}}

\newcounter{chapter}

\newif\if@openright

\renewcommand{\tocsection}[3]{%
  \indentlabel{\@ifnotempty{#2}{\bfseries\ignorespaces#1 #2\quad}}\bfseries#3} 
\renewcommand{\tocsubsection}[3]{%
  \indentlabel{\@ifnotempty{#2}{\ignorespaces#1 #2\quad}}#3}
\newcommand\@dotsep{4.5}
\def\@tocline#1#2#3#4#5#6#7{\relax
  \ifnum #1>\c@tocdepth % then omit
  \else
    \par \addpenalty\@secpenalty\addvspace{#2}%
    \begingroup \hyphenpenalty\@M
    \@ifempty{#4}{%
      \@tempdima\csname r@tocindent\number#1\endcsname\relax
    }{%
      \@tempdima#4\relax
    }%
    \parindent\z@ \leftskip#3\relax \advance\leftskip\@tempdima\relax
    \rightskip\@pnumwidth plus1em \parfillskip-\@pnumwidth
    #5\leavevmode\hskip-\@tempdima{#6}\nobreak
    \leaders\hbox{$\m@th\mkern \@dotsep mu\hbox{.}\mkern \@dotsep mu$}\hfill
    \nobreak
    \hbox to\@pnumwidth{\@tocpagenum{\ifnum#1=1\bfseries\fi#7}}\par% <-- \bfseries for \section page
    \nobreak
    \endgroup
  \fi}
\AtBeginDocument{%
\expandafter\renewcommand\csname r@tocindent0\endcsname{0pt}
}
\def\l@subsection{\@tocline{2}{0pt}{2.5pc}{5pc}{}}
\makeatother

\begin{document}

\title[Pair Dependent Statistics for Arbitrary $\beta$]{Central Limit Theorem for $C\beta E$ Pair Dependent Statistics in Mesoscopic Regime}

\author[A. Aguirre and A. Soshnikov]{Ander Aguirre and Alexander Soshnikov}

\address{University of California at Davis \\ Department of Mathematics \\ 1 Shields Avenue \\Davis CA 95616 \\ United States of America} 

\email{aaguirre@ucdavis.edu}

\address{University of California at Davis \\ Department of Mathematics \\ 1 Shields Avenue \\  Davis CA 95616 \\ United States of America} 

\email{soshniko@math.ucdavis.edu}


 
 \date{ 27 September 2021}

\begin{abstract} We extend our results on the fluctuation of the pair counting statistic of the Circular Beta Ensemble
$\sum_{i\neq j}f(L_N(\theta_i-\theta_j))$
for arbitrary $\beta>0$ in the mesoscopic regime $L_N=\mathcal{O}\of{N^{2/3-\epsilon}}$. In addition, we consider bipartite statistics in the local regime for $\beta=2.$
\end{abstract}

\maketitle

\tableofcontents

\section{Introduction}

     The Circular Beta Ensemble ($C\beta E$) is a random point process of $N\geq 2$ particles on the unit circle; where the joint probability density of the particles $ \theta_j\in[0,2\pi)$, $1\leq j\leq N,$ with respect to the Lebesgue measure is given by:
     
     \begin{align}\label{cbe}p_{\beta,N}(\theta_1,\cdots, \theta_N)=\frac{1}{Z_{\beta,N}}\prod_{j<k}|e^{i\theta_j}-e^{i\theta_k}|^{\beta}.
\end{align} 
Here $\beta>0$ and $Z_{\beta, N}$  is the normalization constant: 
\[Z_{\beta,N}=(2\pi)^{N}\frac{\Gamma(1+\frac{\beta N}{2})}{(\Gamma(1+\frac{\beta}{2}))^N}.\]
The $C\beta E$ generalizes the classical ensembles of random unitary matrices (COE/CUE/CSE) introduced by Dyson in the 1960s in the context of quantum physics (see e.g. \cite{Dyson1}-\cite{Dyson4}). The $C\beta E$ can be interpreted as a Coulomb gas, or system of $N$ repelling particles, with  $\beta$ taking the role of the inverse temperature. It can also be viewed as the limiting invariant distribution of a stochastic evolution process on the eigenvalues known as the \textit{Dyson Brownian motion} (see \cite{webb}).  An explicit sparse random matrix model with eigenvalue distribution matching the $C\beta E$ was introduced in \cite{KN}. To understand the fluctuation of the eigenvalues one can study \textit{linear statistics} of the form
$\sum_{i=1}^Nf(L_N\theta_i), \ 1\leq L_N\leq N.$
The Central Limit Theorem for linear statistics of eigenvalues of the $C\beta E$ was proved 
by Johansson in \cite{johansson1} and extended  in \cite{lambert} beyond the macroscopic regime to $L_N\to \infty$. Recently, in \cite{pairs} and \cite{var} we studied  \textit{pair dependent} statistics of the form:

\begin{align}
  \label{pairs}  S_N(f)=\sum_{i\neq j}f_{L_N}(\theta_i-\theta_j),
\end{align}
where $f_{L_N}(\theta)=f(L_N\*\theta)$ for
$\theta \in [-\pi, \pi)$ and is extended $2\*\pi$-periodically to the whole real line. In the global regime we take $\ L_N=1 \ $ and $f$ to be a sufficiently smooth function on the unit circle. In the mesoscopic regime ($L_N\to \infty, \ \frac{L_N}{N}\to 0$) and the local regime ($L_N=N$) we consider 
$f$ to be a smooth compactly supported function on the real line.

 The research in \cite{pairs}-\cite{var} was motivated by a classical result of Montgomery on pair correlation of zeros of the Riemann zeta function \cite{montgomery1}-\cite{montgomery2}.
Assuming the Riemann Hypothesis, Montgomery studied the distribution of the 
``non-trivial'' zeros on the critical line $\{ 1/2 \pm \gamma_n\}$. In particular, for sufficiently large $T$, fast decaying $f$ with $\operatorname{Supp}{\mathcal{F}(f)}  \subset [-\pi, \pi],$ and rescaling $\tilde{\gamma_n}=\frac{\gamma_n}{2\pi}\log(\gamma_n)$ he considered the statistic:

\[ \sum_{0<\tilde{\gamma}_j\neq \tilde{\gamma}_k <T} f(\tilde{\gamma}_j- \tilde{\gamma}_k ).\]

The results of \cite{montgomery1}-\cite{montgomery2} imply that the two-point correlations of the (rescaled) critical zeros coincide in the limit with the local two point correlations of the eigenvalues of a CUE ($\beta=2$) random matrix. \\

The asymptotic distribution of the pair counting statistic (\ref{pairs}) depends on the speed of the growth of $L_N,$ regularity (smoothness) properties of the test function $f,$ and the value of the inverse temperature
$\beta>0.$ The results of \cite{pairs} deal with the limiting behavior of (\ref{pairs}) in three different regimes, namely macroscopic ($L_N=1$), mesoscopic ($1\ll L_N \ll N$) and microscopic ($L_N=N$).
In the macroscopic (unscaled) $L_N=1$ case it was shown  that $S_N(g)$
has a non-Gaussian fluctuation in the limit $N\to \infty$ provided $g$ is a sufficiently smooth function on the unit circle. In particular (see Theorem 2.1 in \cite{pairs}),

\[S_N(g)-\E S_N(g)\xrightarrow{\hspace{2mm}\mathcal{D}\hspace{2mm}  } \frac{4}{\beta}\sum_{m=1}^{\infty}\hat{g}(m)m(\varphi_m-1)\]
        
where $\varphi_m$ are i.i.d. exponential random variables with  $\E(\varphi_m)=1$,
and 
\begin{align*}
\hat{g}(m)=\frac{1}{2\*\pi}\*\int_0^{2\*\pi}g(x)\* e^{-i\*m\*x}\* dx, \ \ \ m\in \Z,
\end{align*}
are the Fourier coefficients of $g.$
The result was proved under the optimal condition
 $f'\in L^2(\mathbb{T})$ for $\beta=2$, and under slightly sub-optimal conditions for $\beta\neq 2.$\\

In the case of  a slowly growing variance (i.e. when $\sum_{m=-N}^N |\hat{g}(m)|^2\*m^2$ is a slowly growing sequence) the asymptotic fluctuation becomes Gaussian (see \cite{var}). The determinantal structure of the correlation functions of the CUE ($\beta=2$) enabled us to study  the pair counting statistic up to the microscopic regime. In particular, a pair counting statistic was shown to have limiting Gaussian fluctuation provided $f$ is sufficiently smooth. However, for arbitrary $\beta\neq2$ the growth of $L_N$ in \cite{pairs} was restricted to $L_N=o(N^{\varepsilon})$. In this note, we extend the results of \cite{pairs} for $S_N(f)$ and arbitrary $\beta>0$ to $L_N=\mathcal{O}(N^{2/3-\epsilon})$ in the mesoscopic regime. Next, we formulate the main result.\\

\hypertarget{main}{}

\begin{thm} Let $L_N\to \infty$ so that $L_N=\mathcal{O}(N^{2/3 -\epsilon})$,  where $\epsilon>0$ is arbitrary small and  $f\in C_c^{\infty}$ be even, smooth and compactly supported. Consider a pair counting statistic $ S_N(f)$ defined in (\ref{pairs}). Then as $N\rightarrow \infty$

\begin{align}
\label{main}\frac{S_N(f)-\E S_N(f)}{\sqrt{L_N}}\overset{\mathcal{D}}{\longrightarrow} \mathcal{N}\left(0,\frac{4}{\pi \beta^2}\int_{\R}|\hat{f}(t)|^2\*t^2\*dt\right).
\end{align}

\end{thm}
Here $
\hat{f}(t)=\frac{1}{\sqrt{2\*\pi}}\*\int_{\R} f(x)\*e^{-i\*t\*x} \* dx $
denotes the Fourier transform of $f,$
and the notation $\overset{\mathcal{D}}{\longrightarrow}$ denotes convergence in distribution.  We also consider bipartite statistics
\begin{align}
  \label{bipairs}  
  B_N(f)=\sum_{i,j}f_{L_N}(\tau_i-\theta_j),
\end{align}
where $\{\tau_i\}_{i=1}^N$ and $\{\theta_j\}_{j=1}^N$ come from different ensembles on the unit circle, for example two independent $C\beta E$ ensembles. The following result holds.

\begin{thm} Let $\{\tau_i\}_{i=1}^N$ and $\{\theta_j\}_{j=1}^N$ be point configurations from two independent $C\beta E$ ensembles,
$L_N=\mathcal{O}(N^{2/3 -\epsilon})$,  with $\epsilon>0$ is arbitrary small and  $f\in C_c^{\infty}$ be even, smooth and compactly supported. Then $\E B_N(f)=\frac{N^2}{2\*\pi\*L_N}\*\int_{\R} f(x)\*dx$ and

\begin{align}
\label{main}\frac{B_N(f)-\E B_N(f)}{\sqrt{L_N}}\overset{\mathcal{D}}{\longrightarrow} \mathcal{N}\left(0,\frac{2}{\pi \beta^2}\int_{\R}|\hat{f}(t)|^2\*t^2\*dt\right).
\end{align}

\end{thm}

\begin{remark}
The global regime $L_N=1$ for a sufficiently smooth test function $g$ on the unit circle can be treated similarly to Theorem 2.1 in \cite{pairs}.   In particular,  the mean of a bipartite statistic is given by
$\E \sum_{i,j=1}^N g(\tau_i-\theta_j)=\hat{g}(0) \*N^2,$
and
\[B_N(g)-\E B_N(g)\xrightarrow{\hspace{2mm}\mathcal{D}\hspace{2mm}  } \frac{2}{\beta}\sum_{m=1}^{\infty}\hat{g}(m)\*m\*\phi_m,\]
where $\phi_m$ are i.i.d. centered double exponential (Laplace) random variables with $\var(\phi_m)=2.$
\end{remark}

We will denote a $N\times N$ random $C\beta E$ matrix by $U_N$ . The notation $a_N=o(b_N)$ means that $a_N/b_N \to 0$ as $N \to \infty$.
The notation $a_N=\mathcal{O}(b_N)$ means that the ratio of $a_N$ and $b_N$ is bounded in $N.$ 
In section 2, we recall some preliminary material.
Theorems 1.1 and 1.2 will be proved in Section 3. Local bipartite statistics are studied in Section 4. \\

Research has been partially supported  by the Simons Foundation Collaboration Grant for Mathematicians \#312391.

\section{Preliminary Material}
Consider an even real-valued function $g$ on the unit circle that can be represented by the Fourier series
$ g(x)=\sum_{m=-\infty}^{\infty} \hat{g}(m) \* e^{i\*m\*x}.$
Fourier expanding the pair dependent statistic we obtain:
        \begin{align*}
            \sum_{1\leq i\neq j\leq N}g(\theta_i-\theta_j)=2\sum_{m=1}^{\infty}\hat{g}(m)\left|\sum_{j=1}^N\exp\of{im\theta_j}\right|^2+ 
\hat{g}(0)\*N^2-N\*g(0).
        \end{align*}

Let now $f\in C_c^{\infty}(\R)$ be even, smooth and compactly supported, and $L_N\to \infty$ as $N\to \infty.$ Then for sufficiently large $L_N$ we can view $f(L_N\*x)$ as a smooth compactly supported function on the unit circle and the pair counting statistic (\ref{pairs}) can be written as:

\begin{align}
\label{pairs10}
      S_N(f) = \sum_{1\leq j\neq k\leq N} f_{L_N}(\theta_j-\theta_k) = 
 \frac{\hat{f}(0)\*N^2}{\sqrt{2\pi}L_N} +\sum_{k=1}^{\infty} \frac{2}{\sqrt{2\pi}L_N}\hat{f}\of{\frac{k}{L_N}}\left(|\Tr U_N^k|^2-N\right),
\end{align}
where   $ \hat{f} $     denotes the Fourier transform   of $f$. We will use the following notation for traces of powers of a random unitary matrix:
\begin{align}
\label{not} T_N^{(k)}:=\sum_{j=1}^Ne^{ik\theta_j}=\Tr U_N^k, \ \ k=0, \pm 1, \pm2, \ldots. 
\end{align}

Our proof relies on the results of Johansson and Lambert \cite{jl},  who estimated the Wasserstein-2  distance between a random vector of traces of powers of a $C\beta E $ matrix $U_N$ and a random vector of independent Gaussians of matching variance.  We refer to Theorem 1.5 in \cite{jl}. In the Appendix, we justify the claim in Remark 1.1 of \cite{jl} that enables us to extend results to arbitrary $\beta>0$ using \hyperlink{jm}{Proposition 2.3} below. 

Let $T_{d}=\left(T_N^{(k)}\right)_{k=1}^d$ be the vector  of the traces of the first $d$ powers of a random $C\beta E$ matrix $U_N$ and $G_{d}=\left(\sqrt{\frac{2}{\beta}k}Z_k\right)_{k=1}^d, $
where $Z_k$ are i.i.d. complex $\mathcal{N}(0,1)$.  Reformulated in terms of the pair $(T_{d}, G_{d})$, their result states:\\

\hypertarget{jl}{}
\begin{prop}[Johansson and Lambert \cite{jl}] 
Let $2\*d\leq N.$ and $\{e^{i\theta_j}\}_{j=1}^N$ be drawn from the $C\beta E$ with $\beta>0$. Then as $N\rightarrow \infty$ we have the following bound:
\begin{align}
    \label{johlam}\mathcal{W}_2\left(T_{d}, G_{d}\right)=\mathcal{O}\of{\frac{d^2}{N}}.
\end{align}

\end{prop}

We recall that the Wasserstein-$p$ distance between two probability measures on a normed space is defined as (see e.g. \cite{vil}):

\begin{align*}
 \mathcal{W}_p(\mu, \nu):= \left(\inf\{ \E ||X- Y||^p: \ (X,Y) \text{ is  a  r.v.  such  that} \ X\sim \mu, \ Y\sim \nu\}\right)^{1/p},
\end{align*}

where $p\geq 1$ and the notation $X\sim \mu$ means that a random variable $X$ has probability distribution $\mu.$ The Wasserstein distance takes values in $[0, \infty].$ For $p=1,$ the Wasserstein-1 distance $\mathcal{W}_1(\mu, \nu)$ is also known as the Kantorovich-Monge-Rubinstein metric and can be equivalently written as :

  \begin{align}
\label{wass1} \mathcal{W}_1(\mu, \nu):= \sup\{|\int f\*d\mu- \int f \* d\nu|: \ f \ \text{is} \ 1\text{-Lipschitz}\}.
\end{align}  
In other words, the supremum is taken over all real-valued functions $f$ that satisfy $|f(x)-f(y)|\leq d(x,y),$ where $d$ is the metric on the underlying metric space. \\
\begin{remark}
Important earlier results of D\"{o}bler and Stolz \cite{DoSt} and Webb \cite{webb} bounded from above the Wasserstein-1  distance 
$\mathcal{W}_1(T_{d}, G_{d}).$  In particular, it was shown in \cite{DoSt} that for $\beta=2$ 
one has $\mathcal{W}_1(T_{d}, G_{d}) =\mathcal{O}(d^{5/2}/N).$ 
Webb proved in \cite{webb} the bound $\mathcal{W}_1(T_{d}, G_{d}) =\mathcal{O}(d^{7/2}/N)$ for arbitrary $\beta.$
These results are strengthened by (\ref{johlam}) since
\begin{align*}
\mathcal{W}_1(\mu, \nu)\leq d^{1/2}\* \mathcal{W}_2(\mu, \nu).    
\end{align*}
%Regardless, for our purposes the  bound of \hyperlink{jl}{Proposition 2.1} remains optimal.

\end{remark}

Finally, we will require the following bound on the moments of $T_N^{(k)}$.\\

\hypertarget{jm}{}
\begin{prop}[Jiang and Matsumoto \cite{jm}]
Let $\{e^{i\theta_j}\}_{j=1}^N$ be drawn from the $C\beta E$  and $T_N^k$ defined as in (\ref{not}). For $0\leq k\leq N$ we have:
\begin{equation}
\label{jiang}
\E|T_N^{(k)}|^{2m}\leq \left(1+\frac{\left|\frac{2}{\beta}-1\right|}{N-K+\frac{2}{\beta}}\mathbf{1}(\beta>2)\right)^{K}\times \of{\frac{2}{\beta}}^m \times  k^m \times m!
\end{equation}

\noindent where $K=k\*m\leq N$. \\

\end{prop}

\section{Mesoscopic Case}
\begin{proof}[\hyperlink{main}{Proof of Theorem 1.1}] 
Let $L_N=\mathcal{O}\of{N^{2/3-\epsilon}}$ be going to infinity with $N,$  $\epsilon>0$ be arbitrary small, and $\lfloor d=L_N \* N^{\epsilon/2}\rfloor.$ Going forward we can ignore the constant term \[\frac{1}{\sqrt{2\pi}L_N}\hat{f}(0)\*N^2 -N\*f(0)\]
appearing in (\ref{pairs10}) since it disappears upon centralization. Taking into account the smoothness of $f$ we will approximate the pair counting statistic $S_N(f)$ by a truncated version
\begin{align} 
\label{approxim}
 S_{N,d}(f)=\sum_{k=1}^{d} \frac{2}{\sqrt{2\pi}L_N}\hat{f}\of{\frac{k}{L_N}}\*|T_N^{(k)}|^2,\end{align}
and compare its distribution with the distribution of
\begin{align}
\label{gggg} S_d=\sum_{k=1}^{d} \frac{2}{\sqrt{2\pi}L_N}\hat{f}\of{\frac{k}{L_N}}\*\frac{2k}{\beta}\*|Z_k|^2.\end{align}

In Lemma 3.1 we show that the error of the approximation (\ref{approxim}), namely 

\begin{align}
\label{errorapprox}
V_{N,d}(f)=\sum_{k>d} \frac{2}{\sqrt{2\pi}L_N}\hat{f}\of{\frac{k}{L_N}}\*|T_N^{(k)}|^2
\end{align}
is negligible in the limit of large $N.$  In Lemma 3.2., we show that (\ref{gggg}) converges in distribution to a centered Gaussian with variance 
$\frac{4}{\pi \beta^2}\int_{\R}|\hat{f}(\xi)|^2\xi^2d\xi.$ Finally, the main result follows from the Wasserstein distance bound in Lemma 3.3.\\

\hypertarget{lemma2}{}

\begin{lemma} Let  $f\in C_c^{\infty}(\R)$ be even, smooth and compactly supported and further assume that $L_N=o(N^{2/3-\epsilon})$  and $ d= \lfloor L_N\* N^{\epsilon/2}\rfloor.$ Then

\begin{align}
 \label{l24}\frac{V_{N,d}(f)-\E\of{V_{N,d}(f) }}{\sqrt{L_N}} \overset{L^1}{\longrightarrow}0.
\end{align}
\end{lemma}

\begin{proof} From the triangle inequality:

\begin{align}
\label{tri} \E|\text{LHS}(\ref{l24})|&\leq 
 \frac{1}{\sqrt{L_N}}\sum_{k\geq d}\frac{2}{\sqrt{2\pi}L_N}\left|\hat{f}\of{\frac{k}{L_N}}\right|\* \var(T_N^{(k)}),
 \end{align}
 where we recall that $T_N^{(k)}$ denotes $\Tr (U_N^k),$ the trace of the $k$-th power of a random unitary matrix $U_N.$
 The bound of \hyperlink{jm}{Proposition 2.3} gives us that $ \var(T_N^{(k)})\leq Ck$ for $k\leq N/3$. For $k\geq N/3$ we may bound it trivially by $N^2$. 
 Next, we will use the fact that the Fourier transform of a function $f\in C_c^{\infty}(\R)$ is in Schwartz space and thus decays faster than any power. 
 
 \begin{align}
     \text{RHS}(\ref{tri})&\leq \frac{1}{\sqrt{L_N}}\sum_{k\geq d}^{N/3}\frac{2C'k}{\sqrt{2\pi}L_N}\of{\frac{k}{L_N}}^{-\gamma}+\frac{1}{\sqrt{L_N}}\sum_{k> N/3}\frac{2C''N^2}{\sqrt{2\pi}L_N}\of{\frac{k}{L_N}}^{-\gamma}\nonumber \\
    & \vspace{1mm}\nonumber\\
     &=\mathcal{O}\of{\frac{L_N^{\gamma-3/2}}{d^{\gamma-2}}+N^{3-\gamma}\*L_N^{\gamma-3/2}}.\label{llll}
 \end{align}
 
Setting $\gamma=\frac{2}{\epsilon}$ with $\epsilon>0$ sufficiently small we  obtain that the r.h.s. of (\ref{llll}) goes to zero as $N\to \infty.$

\end{proof}

\hypertarget{lemma22}{}

\begin{lemma} We have the following convergence in distribution:
\begin{align}
\frac{S_d-\E S_d}{\sqrt{L_N}}\overset{\mathcal{D}}{\longrightarrow} \mathcal{N}\left(0,\frac{4}{\pi \beta^2}\int_{\R}|\hat{f}(\xi)|^2\xi^2d\xi\right).
\end{align}
\end{lemma}

\begin{proof}

For $k\geq1$ let us denote 
\begin{align}
\label{ak}a_k=\frac{2}{\sqrt{2\pi}L_N}\hat{f}\of{\frac{k}{L_N}}\*\frac{2k}{\beta}, \ \ \  \quad \quad \quad  S_d=\sum_{k=1}^{d} a_k\*|Z_k|^2.\end{align}
Then we have 
\begin{align*}
\frac{\var(S_d)}{L_N}=\frac{1}{L_N}\times\sum_{k=1}^d   \frac{8}{\pi \beta^2}  \*\left|\hat{f}\of{\frac{k}{L_N}}\right|^2\* \left|\frac{k}{L_N}\right|^2 \longrightarrow \frac{4}{\pi \beta^2}\int_{\R}|\hat{f}(\xi)|^2\xi^2d\xi.
\end{align*}
Since $\max\{a_k^2\}_{k=1}^d=o(\sum_1^d a_k^2)$, the result follows from the Lindeberg-Feller theorem.

\end{proof}

\begin{lemma}
Let $d\leq N/2.$ Then we have the following Wasserstein-1 bound:
\begin{align}
\label{wasserwasser}
\mathcal{W}_1\left(\frac{S_{N,d}}{L_N^{1/2}}, \frac{S_{d}}{L_N^{1/2}}\right)=\mathcal{O}\of{\frac{d^{3}}{N\*L_N^{3/2}}},
\end{align}
Therefore L.H.S(\ref{wasserwasser})$\rightarrow 0$ if $L_N=\mathcal{O}\of{N^{2/3-\epsilon}}$
and $d=\lfloor L_N\*N^{\epsilon}\rfloor.$
\end{lemma}

\begin{proof}
Let $(\cdot,\cdot)$  denote standard Euclidean inner product in $\C^d$ and $A$ be a diagonal $d\times d$ matrix with $A_{k,k}=a_k, \ 1\leq k\leq d,$ with $a_k$ as in (\ref{ak}).  With $T_d$ and $G_d$ as defined in \hyperlink{jl}{Proposition 2.1} we have:

\begin{align*}
  L_N^{-1/2}\*|S_{N,d}-S_d|&=L_N^{-1/2}\* |(A\*T_d, T_d)-(A\*G_d,G_d)|  \\
  & \leq L_N^{-1/2}\*|(A\*T_d, T_d-G_d) + (A(T_d-G_d), G_d)|\\
  & \leq L_N^{-1/2}\* ||A||\times\left\lVert T_d-G_d\right\rVert_2^2 + 2\*L_N^{-1/2}\*||A||\times\left\lVert T_d-G_d\right\rVert_2\times\left\lVert G_d\right\rVert_2.
    \end{align*}
    Here $||A||$ denotes the operator norm of $A$ and $\left\lVert \cdot \right\rVert_2$ is the vector $l^2$ norm.

 Since A is a diagonal matrix, one has $||A||=\max\{|a_k|, \ 1\leq k\leq d\}=\mathcal{O}\of{L_N^{-1}}.$  \hyperlink{jl}{Proposition 2.2} allows us to choose the components of the vector $G_d$ so that
\begin{align}
    \label{rvs1} \E (\left\lVert T_d-G_d\right\rVert_2^2)=\mathcal{O}\of{\frac{d^{4}}{N^2}}.
 \end{align}   
Using $\E (\left\lVert G_d\right\rVert_2^2) =   \mathcal{O}\of{d^2}$ and the Cauchy-Schwartz inequality we arrive at:

\begin{align*}
  L_N^{-1/2}\* \E   |S_{N,d}-S_d|&\leq L_N^{-3/2}\*\E (\left\lVert T_d-G_d\right\rVert_2^2) +2\* L_N^{-3/2}\*\E (\left\lVert T_d-G_d\right\rVert_2\times \left\lVert G_d\right\rVert_2)\\
 & \leq \mathcal{O}\of{\frac{d^{4}}{N^2\*L_N^{3/2}}} +\mathcal{O}\of{\frac{d^{3}}{N\* L_N^{3/2}}},
\end{align*}

\end{proof}

Combining the statements of Lemmas 3.1, 3.2, and 3.3 we finish the proof of the Theorem 1.1.
\end{proof}
The proof of Theorem 1.2 is quite similar with minor changes (such as replacing the approximation of the quadratic form $(A\*T_d, T_d)$ by the approximation of the bilinear form $(A\*T_d, \mathcal{T}_d),$
where $\mathcal{T}_d$ is an independent copy of $T_d.$ The details are left to the reader.

\section{Local Bipartite Statistics}
In this section we study bipartite statistics in the local regime $L_N=N$ for $\beta=2.$
We recall that in Theorem 2.5 in \cite{pairs}
we considered
$S_N(f)=\sum_{1\leq i\neq j\leq N} f_N(\theta_i-\theta_j),$
$ f_N(\theta)=f(N\*\theta)$ for $\theta\in [-\pi, \pi), \ \ f_N(\theta+2\*\pi)=f_N(\theta),$ and proved
\begin{thm}
Let $\beta=2$ and $f \in C^{\infty}_c(\R)$ be an even, smooth, compactly supported function on the real line. 
Then $(S_N(f) -\E S_N(f))\*N^{-1/2}$ converges in distribution to centered real Gaussian random variable with the variance
\begin{align}
\label{vvvv}
\frac{1}{\pi}\* \int_{\mathbb{R}} |\hat{f}(t)|^2 \*\min(|t|,1)^2\* dt -\frac{1}{\pi}\* \int_{|s-t|\leq 1, |s|\vee|t|\geq 1} \hat{f}(t)\*\hat{f}(s)\*
(1-|s-t|)\* ds\*dt\\
-\frac{1}{\pi}\* \int_{0\leq s,t\leq 1, s+t>1} \hat{f}(s)\*\hat{f}(t)\*(s+t-1) \*ds\*dt. \nonumber
\end{align}
\end{thm}
Consider
\begin{align}
\label{bilocal}
B_N(f)=\sum_{1\leq i, j\leq N} f_N(\tau_i-\theta_j),
\end{align}
where the point configuration
$\{\tau_i\}_{i=1}^N$ comes from the following three ensembles: (i) an independent copy of $C\beta E,$
(ii) a collection of i.i.d. uniformly distributed points on the unit circle, (iii) evenly spaced deterministic sequence. The following result holds:
\hypertarget{fourtwo}{}

\begin{thm}
Let $f \in C^{\infty}_c(\R)$ be an even, smooth, compactly supported function on the real line. Consider $B_N(f)$ defined in (\ref{bilocal}) where $\{\theta_j\}_{j=1}^N$ be a CUE configuration and
$\{\tau_i\}_{i=1}^N$ comes from one of the following three ensembles:
\begin{itemize}
    \item[(i)] an independent copy of a CUE; 
\item[(ii)] a sequence of i.i.d. uniformly distributed points on the unit circle;
\item[(iii)] an evenly spaced deterministic sequence $\tau_i=\frac{2\*\pi\*i}{N}, \ i=1, \ldots, N.$\\
\end{itemize}

Then $\E B_N(f)=\frac{N}{2\*\pi}\*\int_{\R} f(x)\*dx,$ and $(B_N(f) -\E B_N(f))\*N^{-1/2}$ converges in distribution to centered real Gaussian random variable with variance $\sigma^2(f),$ where
\begin{align}
\label{Dispersiya}
\sigma^2(f)= \begin{cases} \frac{1}{2\*\pi}\* \int_{\mathbb{R}} |\hat{f}(t)|^2 \*\min(|t|,1)^2\* dt  &\text{in the case  }  (i),\\
\frac{1}{2\*\pi}\* \int_{\mathbb{R}} |\hat{f}(t)|^2 \*\min(|t|,1)\* dt &\text{in the case  }  (ii),\\
\frac{1}{2\*\pi}\*\sum_{l\neq 0}|\hat{f}(l)|^2 &\text{in the case  }  (iii).
\end{cases}
\end{align}
\end{thm}
\begin{proof}[\hyperlink{fourtwo}{Proof of Theorem 4.2}] 

As always 
$T_N^{(k)}:=\sum_{j=1}^Ne^{ik\theta_j}$  and denote
$\mathcal{T}_N^{(k)}= \sum_{j=1}^Ne^{ik\tau_j},$ where
$\ \ k\in \Z.$ Then
\begin{align}
B_N^c(f):=B_N(f)-\E B_N(f)= \sum_{k\neq 0} \frac{1}{\sqrt{2\pi}\*N}\hat{f}\of{\frac{k}{N}}\*T_N^{(k)} \* \mathcal{T}_N^{(-k)}. 
    \end{align}
We consider the case (i) first. Using independence and $\E |T_N^{(k)}|^2= \min(|k|, N),$ has

\begin{align*}
\var(B_N(f))=  \sum_{k\neq 0} \frac{1}{2\*\pi\*N^2} \left|\hat{f}\of{\frac{k}{N}}\right|^2 \* \min(|k|, N)^2 =
\frac{N}{2\*\pi}\* \int_{\mathbb{R}} |\hat{f}(t)|^2 \*\min(|t|,1)^2\* dt \*(1+o(1)).
\end{align*}

To study higher moments, we use cumulant bounds and power counting. One writes for $l\geq 2$
\begin{align}
\label{centralmoments}
\E (B_N^c(f))^l=(2\*\pi)^{-l/2}\*N^{-l}\*\sum_{k_1,\ldots ,k_l\neq 0} \prod_{i=1}^l \hat{f}\of{\frac{k_i}{N}}
\*\E \prod_{i=1}^l T_N^{(k_i)}\*\E \prod_{j=1}^l \mathcal{T}_N^{(-k_j)}.
\end{align}
To evaluate the moments we use Lemma 5.2 from \cite{pairs} that allows one to estimate joint cumulants of
the traces of powers .
It was shown that for any $n\geq 1, \ \kappa_n^{(N)}(k_1, \ldots, k_n),$ the $n$-th joint cumulant of $T_N^{(k)}$'s is $O(N),$ uniformly in $k_1, \ldots k_n.$ In addition,
$\kappa_2^{(N)}(k_1, k_2)=\min(N, |k_1|)\*1_{k_1+k_2=0}.$ This implies that for odd values of $l=2m+1$

\begin{align*}
\E (B_N^c(f))^{2m+1}=O(N^m),    
\end{align*}

and for even values $l=2m$ main contribution to (\ref{centralmoments}) comes from the $l$-tuples 
$(k_1, \ldots, k_m)$ that could be split into pairs $(k,-k).$  By power counting one then obtains

\begin{align}
\E (B_N^c(f))^{2m}=\sigma^{2m}\*(2m-1)!!\*N^m\*(1+o(1)),    
\end{align}

and the moment convergence implies CLT. The considerations in the case (ii) are very similar. In particular,

\begin{align*}
\var(B_N(f))=  \sum_{k\neq 0} \frac{1}{2\*\pi\*N^2} \left|\hat{f}\of{\frac{k}{N}}\right|^2 \* \min(|k|, N)\*N =
\frac{N}{2\*\pi}\* \int_{\mathbb{R}} |\hat{f}(t)|^2 \*\min(|t|,1)\* dt \*(1+o(1)).
\end{align*}

We leave higher order estimates to the reader. Finally, we turn our attention to the case (iii). In this case we have:

\begin{align}
B_N^c(f)= \sum_{l\neq 0} \frac{1}{\sqrt{2\pi}}\hat{f}(l)\*T_N^{(l\*N)}. 
    \end{align}
This readily implies

\begin{align}
\var(B_N(f))=  N\* \frac{1}{2\*\pi}\*\sum_{l\neq 0}|\hat{f}(l)|^2.
\end{align}
The Central Limit Theorem again follows from the cumulant bounds and power counting. It should be noted that 
random variables  $T_N^{(l\*N)}, \ \ l\in \Z\setminus\{0\}, $ are not independent but are identically distributed - they have the same distribution as $\sum_{j=1}^N e^{i\*\tau_j}.$

\end{proof}
\section{Appendix}

In this appendix we discuss the details in Remark 1.1 of \cite{jl} that justify the statement of \hyperlink{jl}{Proposition 2.1} for arbitrary $\beta>0$. In Theorem 1.5 of \cite{jl} Johnasson and Lambert  provide the bound:
\begin{equation}
    \mathcal{W}_2\of{X,G}\leq \mathcal{O}\of{\frac{d^{3/2}}{N}},
\end{equation}

where
\begin{equation}
    X=\sqrt{\frac{\beta}{2 k}} T_d\quad \quad G= \sqrt{\frac{\beta}{2 k}} G_d.
\end{equation}
 Now define the map $X:\T^n\rightarrow \R^{2d}$ by $X_{2k-1}=\Re T_N^{(k)}$ and  $X_{2k}=\Re T_N^{(k)}, \ \ 1\leq k \leq d,$ and set $\mathbf{\Gamma}$ to be a $2d$-dimensional square matrix with the entries $\mathbf{\Gamma}_{k,l}=\nabla X_k \cdot \nabla X_l$
 
 We also define
\[\begin{aligned}
\mathbf{K}  & = N \cdot \mathrm{diag}(1,1,2,2,\cdots, d,d) \\
\xi & = \big( \Re \zeta_1  ,  \Im \zeta_1   ,  \Re \zeta_2 , \Im \zeta_2 , \cdots ,  \Re \zeta_d  ,  \Im \zeta_d \big),
\end{aligned}\]
where for $k\ge 1$ 
\vspace*{-.3cm}
\[
\zeta_k =   \sqrt{\frac k2} \sum_{\ell = 1}^{k-1} \sqrt{\ell(k-\ell)}  \mathrm{T}_N^{(\ell)} \mathrm{T}_N^{(k-\ell)} . 
\]

We refer to Section 7 of \cite{jl} (specifically Lemmas 7.2 and 7.3) for full details of the following lemma.

\begin{lemma}
For all $N,d\in\N$ and for any positive definite diagonal matrix $\mathbf{K}$ of size $2d \times 2d$, we have
\begin{equation} \label{LLW}
\mathcal{W}_2(X,G) 
\le \sqrt{\E_N \big[ |  \mathbf{K}^{-1} \xi |^2\big]}
+ \sqrt{  \E_N\big[  \| \mathbf{I} -   \mathbf{K}^{-1} \boldsymbol{\Gamma} \|^2 \big] } , 
\end{equation}
where $\| \cdot\|_{HS}$ denotes the Hilbert--Schmidt norm. 
\end{lemma}

We arrive at the desired bound with the following lemma. This corresponds to lemmas 7.4 and 7.5 of \cite{jl} where we used the moment estimates of Jiang-Matsumoto instead. \\

\begin{lemma}
With $\mathbf{K}, \mathbf{\Gamma}$ and $\xi$ as above we have:
\begin{equation}
   \quad \quad \quad \E_n |  \mathbf{K}^{-1} \xi|^2=\mathcal{O}\of{\frac{d^3}{N^2}}   \quad \quad \quad \E_n[||\mathbf{I-K^{-1}\Gamma}||_{HS}^2]=\mathcal{O}\of{\frac{d^3}{N^2}}  .
\end{equation}

\end{lemma}

\begin{proof} From  page 37 of \cite{jl} we have the identity:

\begin{equation*}
    ||\mathbf{I-K^{-1}\Gamma}||_{HS}^2=\frac{\beta}{N^2}\sum_{1\leq k<l\leq d}|T_{N}^{(l-k)}|^2+|T_{N}^{(l+k)}|^2\ \ +\frac{5\beta}{2N^2}\sum_{k=1}^d|\Re(T_{N}^{(k)})|^2\end{equation*}

   Since  $k\leq L_N$, \hyperlink{jm}{Proposition 2.3 } gives that $\E |T_N^{(k)}|^2\leq C_{\beta}k$ and thus

    \begin{equation*} 
\E_n\big[  \| \mathbf{I} -  \mathbf{K}^{-1} \boldsymbol{\Gamma} \|_{HS}^2 \big]   
\leq C_{\beta}\times   \left(\sum_{1\le k < \ell \le d} \frac{\ell}{N^2} + \sum_{k=1}^d  \frac{k}{N^2} \
\right)=\mathcal{O}\of{\frac{d^3}{N^2}} .
\end{equation*}
    
From page 36 of \cite{jl}  we have the identity:
 \begin{equation}
|  \mathbf{K}^{-1} \xi|^2  = \sum_{1\le \ell , \ell' < k \le d}     \frac{\beta^2}{8k N^2}  \mathrm{T}_N^{(\ell)} \mathrm{T}_N^{(k-\ell)}  \overline{\mathrm{T}_N^{(\ell')} \mathrm{T}_N^{(k-\ell')}} . 
\end{equation}
   Since  $k\leq L_N$, \hyperlink{jm}{Proposition 2.3} gives us (see Theorem 1 part (b) of \cite{jm}): 
\[ \begin{aligned}
\E_n\big[  \mathrm{T}_N^{(\ell)} \mathrm{T}_N^{(k-\ell)}  \overline{\mathrm{T}_N^{(\ell')} \mathrm{T}_N^{(k-\ell')}} \big] \
\leq C_{\beta} \left( \ell(k-\ell)\1_{\{\ell = \ell',\ell\neq k/2\}} + \ell(k-\ell)\1_{\{\ell = k- \ell',\ell\neq k/2\}} )  +  \frac{k^2}{4}\1_{\ell = \ell' = k/2}  \right).
\end{aligned}\]

Taking expectations we have:
\[  
\E_N \big[ |  \mathbf{K}^{-1} \xi|^2 \big]  \leq  \frac{K_{\beta}}{N^2} \sum_{1\le \ell  < k \le d} \hspace{-.3cm}   \frac{\ell(k-\ell)}{ k} 
= \mathcal{O}\of{\frac{d^3}{N^2}}.\]

\end{proof}

\end{document}